\documentclass[12pt,a4paper,
fleqn,reqno]{amsart}                 
\usepackage[T1]{fontenc}                   
\usepackage[utf8]{inputenc}                 
\usepackage[english]{babel}    
 \usepackage{csquotes}
\usepackage{graphicx}                      %
\usepackage[font=small]{quoting}    
\usepackage{float}
\usepackage{caption}  
\usepackage[top=1.2in,bottom=1.8in,left=1.in,right=1.in]{geometry}                  
\usepackage{verbatim}
\usepackage{color}
\usepackage{xcolor}
\usepackage{picture}
\usepackage{amsmath,amsthm,amsfonts,amssymb}
\usepackage{comment}
\usepackage{hyperref}
\hypersetup{colorlinks,linkcolor={blue}, citecolor={red},urlcolor={blue}}  
\usepackage{enumitem}
\usepackage{mathtools}
\usepackage[url=false,doi=false,isbn=false, giveninits=true, 
style=numeric,maxbibnames=99]{biblatex}
\newtheorem{thm}{Theorem}[section] 
\newtheorem{lem}[thm]{Lemma} 
 
\newtheorem{prop}[thm]{Proposition} 
 
\newtheorem{defn}{Definition}[section]
\theoremstyle{definition} 
 
\theoremstyle{remark}

\theoremstyle{definition}

\def\O{\Omega}

\def\S{\Sigma} 
\def\n{\nabla}

\def\p{\partial}

\def\a{\alpha}
\def\b{\beta}
\def\n{\nabla}

\def\O{\Omega}
\def\p{\partial}

\def\a{\alpha}
\def\b{\beta}
\def\g{\gamma}
\def\d{\delta}
\def\k{\kappa}
\def\l{\lambda}
\def\s{\sigma}

\def\ov{\overline}

\def\n{\nabla}
\def\<{\langle}
\def\>{\rangle}

\def\n{\nabla}

\def\NN{\mathbb{N}}

\def\RR{\mathbb{R}}

\def\SS{\mathbb{S}}

\def\O{\Omega}
\def\p{\partial}

\def\a{\alpha}
\def\b{\beta}
\def\g{\gamma}
\def\d{\delta}
\def\l{\lambda}

\def\s{\sigma}

\def\ov{\overline}

\def\wh{\widehat}

\def\R{\mathbb{R}}

\def\C{\mathcal{C}}

\def\C{\mathcal{C}}

\def\ol{\overline}
\def\wt{\widetilde}

{\left\{\begin{array}{@{}l@{}}}{\end{array}\right.}
\patchcmd{\abstract}{\scshape\abstractname}{\textbf{\abstractname}}{}{}
\makeatletter 
\def\@makefnmark{} 
\makeatother

\numberwithin{equation}{section}
\numberwithin{exa}{section}

\makeatletter
\@namedef{subjclassname@2020}{\textup{2020} Mathematics Subject Classification}
\makeatother

\begin{document}
\title[Convex capillary hypersurfaces of prescribed curvature problem] {Convex capillary hypersurfaces of prescribed curvature problem}

\author[X. Mei]{Xinqun Mei}
\address[X. Mei]{Key Laboratory of Pure and Applied Mathematics, School of Mathematical Sciences, Peking University,  Beijing, 100871, P.R. China}

\email{\href{qunmath@pku.edu.cn}{qunmath@pku.edu.cn}}

\author[G. Wang]{Guofang Wang}
\address[G. Wang]{Mathematisches Institut, Albert-Ludwigs-Universit\"{a}t Freiburg, Freiburg im Breisgau, 79104, Germany}
\email{\href{guofang.wang@math.uni-freiburg.de}{guofang.wang@math.uni-freiburg.de}}

\author[L. Weng]{Liangjun Weng}
\address[L. Weng]{Centro di Ricerca Matematica Ennio De Giorgi, Scuola Normale Superiore, Pisa, 56126, Italy \& Dipartimento di Matematica, Universit\`a di Pisa, Pisa, 56127, Italy}
\email{\href{mailto:liangjun.weng@sns.it}  {liangjun.weng@sns.it}}



\begin{abstract}
In this paper, we study the prescribed $k$-th Weingarten curvature problem for convex capillary hypersurfaces in $\overline{\mathbb{R}^{n+1}_{+}}$. This problem naturally extends the prescribed $k$-th Weingarten curvature problem for closed convex hypersurfaces, previously investigated by Guan-Guan in \cite{GG02}, to the capillary setting. We reformulate the problem as the solvability of a Hessian quotient equation with a Robin boundary condition on a spherical cap. Under a natural sufficient condition, we establish the existence of a strictly convex capillary hypersurface with the prescribed $k$-th Weingarten curvature. This also extends our recent work on the capillary Minkowski problem in \cite{MWW-AIM}.  
\end{abstract}

\maketitle

\section{Introduction}

\subsection{Setup and motivation}
  Let $\RR^{n+1}_+=\{x\in \R^{n+1}\,|\, x_{n+1}>0\}$ be the upper half-space, and $\S$  a properly embedded, smooth compact hypersurface with boundary in $\overline{\RR^{n+1}_+}$, such that $$\text{int}(\S)\subset \RR^{n+1}_{+} ~~\text{and}~~ \p \Sigma\subset \p \RR^{n+1}_{+}.$$ $\S\subset \ov{\RR^{n+1}_+}$  is called a {\it capillary hypersurface} if $\S$ intersects $\p\RR^{n+1}_+$ at a constant contact angle $\theta\in (0, \pi)$ along $\p\S$. In particular, if $\theta=\frac\pi 2$, $\S$ is known as  \textit{free boundary hypersurface}. Denote $\nu$ as the unit outward normal of $\S$ in $\ol{\RR^{n+1}_+}$, 
  the contact angle $\theta $ is defined by $$ \cos (\pi -\theta)=\langle \nu, e \rangle,\quad \rm{along~\partial\S},$$
	where $e\coloneqq -E_{n+1}=(0, \cdots, 0,-1)$, $E_{n+1}$ is the $(n+1)$-th unit coordinate vector in $\ov{\RR^{n+1}_+}$ and hence $e$ is the unit outward normal of $\p \ov{\R^{n+1}_+}$.
If $\Sigma$ is strictly convex,  it is clear that the image of the Gauss map, namely,  $\nu (\Sigma)$, lies on the spherical cap 
	\begin{eqnarray*}
		\SS^n_\theta \coloneqq \left\{ x\in \SS^n \mid  ~ \langle x, E_{n+1} \rangle \geq \cos \theta \right\},
	\end{eqnarray*} see, e.g., \cite[Lemma 2.2]{MWWX}. Instead of using the usual Gauss map $\nu$, it is more convenient to use the following map
	$$\widetilde \nu \coloneqq  T \circ \nu: \S \to \C_\theta,$$
	where $\C_\theta$ is a unit spherical cap centered at $\cos\theta e$, namely
	\begin{eqnarray*}
		\C_{\theta } = \left\{\xi\in \ov{\mathbb{R}^{n+1}_+}\mid ~|\xi-\cos\theta e|= 1 \right\},
	\end{eqnarray*}
	and $T:\SS^n_\theta\to\C_\theta$ is defined by $ T(y)= y+\cos\theta e$,  which is a translation along the vertical direction.  As in \cite{MWW-AIM}, we refer to $\wt \nu$  as the {\it capillary Gauss map} of $\Sigma$.  It turns out that $\wt \nu: \S\to \C_\theta$ is a diffeomorphism. Therefore,  we can reparametrize $\S$ by its inverse on $\C_\theta$. Moreover, the support function $h$ for $\S$ can be regarded as a function defined on $\C_\theta$, which is given by 
\begin{eqnarray*}
    h(\xi)\coloneqq \left\<\xi-\cos\theta e,\wt\nu^{-1}(\xi)\right\>, ~~\xi\in\C_\theta.
\end{eqnarray*}
One can check that $h$ satisfies
\begin{eqnarray*}
    \n_\mu h=\cot\theta h,~~~~\text{ on } ~~ \p \C_\theta,
\end{eqnarray*}where $\mu$ is the outward  unit co-normal of $\p\C_\theta\subset \C_\theta$. For more details about the geometry of the capillary hypersurface, we refer to \cite[Section 2]{MWW-AIM} and  \cite[Section 2]{MWWX}. It is a natural 
 question to ask

\

{\it
When certain data is prescribed on $\C_\theta$, how much information can be recovered through the inverse capillary Gauss map $\wt\nu$ ? } 

\

An important example is the Gauss-Kronecker curvature function defined on $\C_\theta$, which is known as the capillary Minkowski problem and has been investigated by the authors in \cite{MWW-AIM}.  In addition to the Gauss-Kronecker curvature, the $k$-th Weingarten curvature, denoted $W_k$, plays a significant role in differential geometry. It is defined as
\begin{eqnarray*} 
  W_k(X)\coloneqq \s_k(\k(X)),~~~~X\in \S,
\end{eqnarray*}
where $\kappa(X)\coloneqq (\kappa_{1},\cdots, \kappa_{n})$ are the principal curvatures of $\S$ at $X$, and  $\sigma_{k}(\kappa)$ is the $k$-th elementary symmetric function. When $k=1,2$ and $n$, $W_k$ corresponds to the mean curvature, scalar curvature, and Gauss-Kronecker curvature of $\S$ respectively.  In this paper, we further investigate a prescribed $k$-th Weingarten curvature problem of convex capillary hypersurfaces in $\ov{\RR^{n+1}_{+}}$.  The problem is precisely formulated as follows:

\
	
{\noindent	\textbf{Problem:}\textit{ Let $1\leq k<n$, and $f:\C_{\theta}\rightarrow \RR$ be a positive smooth function. Does there exist a strictly convex capillary hypersurface $\S\subset \ov{\RR^{n+1}_{+}}$  such that 
\begin{eqnarray}\label{cur equ}
 W_k \left(\wt\nu^{-1}(\xi) \right)=f(\xi), ~~~~\xi\in \C_\theta ~? 
\end{eqnarray}}


\

In particular, when $\theta=\frac{\pi}{2}$,  by a simple reflection argument, this problem reduces to the classical prescribed $k$-th Weingarten curvature problem for closed convex hypersurfaces.  The case $k=n$ is known as the Minkowski problem, which is related to the prescribed Gauss-Kronecker curvature problem. Thanks to the contributions of Minkowski \cite{Min}, Alexandrov \cite{Alex}, Caffarelli \cite{Caff1990-1, Caff1990-2}, Lewy \cite{Lewy}, Nirenberg \cite{Nire}, Pogorelov \cite{Pog52}, Cheng-Yau \cite{CY76} and many others,  this problem has been completely solved under a necessary and sufficient condition. For the cases where $1\leq k\leq n-1$, the problem of prescribed $k$-th Weingarten curvature turns out to be more subtle and shows a significant difference to the case  $k=n$. In \cite{GG02}, Guan-Guan proved the existence of a strictly convex hypersurface that satisfies Eq.  \eqref{cur equ} when $f$ is invariant under an automorphic group $G$ without fixed points (for example, if $f$ is an even function on $\SS^n$). Subsequently, 
 Sheng-Trudinger-Wang \cite{STW} removed the evenness assumption albeit at the cost of introducing an exponential weight factor into the equation. For more related works; see, e.g., \cite{BIS2023, CNS-4, Chern, Galvez-Mira2016, Ger, Ger-1, Ger-2, Ger2007, GLM06,  GLM, GM2005, GZ, MWW-1} and references therein.

Compared to the closed hypersurface setting, the prescribed curvature problem for convex hypersurfaces with non-empty boundaries or non-compact convex hypersurfaces has been less extensively developed, except some works for the Dirichlet boundary value problem, see, e.g. \cite{Busemann, CW95, GS, Nehring1998crelle, O, Sch1} for the prescribed Gauss-Kronecker curvature and \cite{Ivochkina-Tomi1998CVPDE, nehring1999agag, Schn} for the prescribed  Weingarten curvature resp. 
Recently, the authors \cite{MWW-AIM} studied the capillary Minkowski problem,  which aims to determine the existence of a convex capillary hypersurface with a prescribed Gauss-Kronecker curvature. This problem reduces to a Robin boundary value problem.   We established a necessary and sufficient condition (i.e., \eqref{nece suff cond}) for the solvability of this problem, analogous to the classical Minkowski problem. The main result of \cite[Theorem 1.1]{MWW-AIM} is stated as follows. 

\

\noindent{\bf Theorem A}. {\it
Let $\theta\in (0,\frac{\pi}{2}]$, and $f\in C^2(\C_\theta)$ be a positive function that  satisfies 
\begin{eqnarray}\label{nece suff cond}
   \int_{\C_{\theta}}\<\xi, E_{\alpha}\>f^{-1} dA_\s =0,\quad \forall 1\leq \alpha\leq n,
\end{eqnarray}
where $\{E_{\a}\}_{\a=1}^n$ are the horizontal unit coordinate vectors in $\overline{\RR^{n+1}_+}$ and $dA_\s$ is the standard area element on $\C_\theta$ with respect to the round metric $\s$. Then there exists a $C^{3,\gamma}$ $(\gamma\in(0,1))$ strictly convex capillary hypersurface $\S\subset \ov{\RR^{n+1}_+}$ such that its Gauss-Kronecker curvature $K$ satisfies $$K \left(\widetilde \nu^{-1}(\xi) \right)= {f(\xi)}, ~~~\xi\in \C_\theta.$$ 
  Moreover, $\S$ is unique up to a horizontal translation in $\ov{\RR^{n+1}_+}$.}

 \

{Theorem A} addresses the case $k=n$ in the aforementioned problem \eqref{cur equ}. The primary goal of this article is to investigate the remaining cases, namely, $1\leq k\leq n-1$. These cases present substantial differences and challenges compared to the case $k=n$, as will be discussed later, or see \cite[Section 1]{GG02}. To state the main result concisely, we introduce a definition that serves as a capillary adaptation of the even function on $\SS^n$.
\begin{defn}
    Let $f\in C^{2}(\C_{\theta})$. For any point $\xi=(\xi_{1},\dots, \xi_{n},\xi_{n+1})\in\C_{\theta}$, we denote $\widehat{\xi}\coloneqq (-\xi_{1},\dots, -\xi_{n}, \xi_{n+1})$. If $f$ satisfies $f(\xi)=f(\widehat{\xi})$, we say that $f$ is a capillary even function on $\C_{\theta}$.
\end{defn}
\subsection{Main results}
The following theorem presents the main result of this paper.
\begin{thm}\label{main-thm}
    Let $\theta\in (0, \frac{\pi}{2}]$ and $1\leq k\leq n-1$. Assume that $f$ is a positive smooth  capillary even function on $\C_{\theta}$. Then there exists a strictly convex capillary hypersurface $\S$ in $\overline{\RR^{n+1}_{+}}$ that satisfies Eq. \eqref{cur equ}. 
\end{thm}

We remark that the restriction on angle  $\theta\in (0, \frac{\pi}{2}]$ is crucially used in two places. The first one is in Lemma \ref{chou-wang-lemma}, where this restriction guarantees that  Chou-Wang's geometric Lemma holds.  The second place is in deriving the $C^2$-estimate.  
See Section \ref{sec3} for more details. One may understand the condition as a convexity of the domain ${\mathcal C}_\theta$. The convexity of domains plays an essential role in addressing Dirichlet or Neumann boundary problems to  fully nonlinear PDEs, see, e.g., \cite{CNS-1, LTU, MQ} among many other results. Nevertheless, we expect that the result still holds for $\theta > \frac{\pi}{2}$.

As a byproduct, we can establish the uniqueness of the solution to the problem \eqref{cur equ}.
This holds when $f$ is a capillary even function and is close to $1$ in the $C^{\alpha}$ norm.

\begin{thm}\label{thm-1.2}
    Let $\theta\in (0, \frac{\pi}{2}]$ and $1\leq k\leq n-1$. Suppose $f$ is a positive, smooth, capillary even function on $\C_{\theta}$. Then there exists a small positive constant $\varepsilon_{0}$ such that if 
    $$\|f^{-1}-1\|_{C^{\alpha}(\C_{\theta})}\leq \varepsilon_{0},$$ the solution to Eq. \eqref{cur equ} is unique.  
\end{thm}

For solving the capillary Minkowski problem, condition \eqref{nece suff cond} is both necessary and sufficient, as established by Theorem A. This condition also serves as a necessary condition for the solvability of the capillary Christoffel-Minkowski problem (cf. \cite[Proposition 2.6]{MWW-CM}), which concerns the existence of the capillary convex body with a prescribed $k$-th capillary area measure. However, by employing the approach from \cite[Section 4]{GG02}, we show that this condition is not sufficient for the solvability of Eq. \eqref{cur equ}. Specifically, we establish the following theorem.
\begin{thm}\label{thm-example}
    For any $1\leq k\leq n-1$, there exists a  one-parameter family of capillary strictly convex hypersurfaces in $\ov{\RR^{n+1}_{+}}$ satisfying 
        \begin{eqnarray*}
            \int_{\C_{\theta}}\frac{\<\xi, E_{\alpha}\>}{W_{k}(\widetilde{\nu}^{-1}(\xi))}dA_\sigma\neq 0, \quad 1\leq \alpha\leq n.
        \end{eqnarray*}
\end{thm}

It would be interesting to ask if condition \eqref{nece suff cond} is a sufficient condition for the solvability of Eq. \eqref{cur equ}. 
The primary challenge lies in establishing Theorem \ref{thm est} below without the capillary even assumption. If this obstacle can be overcome, following a similar argument as in \cite[Section 4]{GG02}, we can then show that condition \eqref{nece suff cond} is not a sufficient condition.


\subsection{Outline of the proof}
To start with, by adopting the arguments presented in \cite[Proposition 2.4]{MWW-AIM} (see also \cite[Lemma~2.4, Proposition~2.6]{MWWX}), Eq. \eqref{cur equ} is equivalent to  solving  the following Hessian quotient equation with a Robin boundary condition, 
\begin{eqnarray}\label{sup equ} \left\{
\begin{array}{rcll}\vspace{2mm}\displaystyle
	 \frac{\sigma_{n}(\n^{2}h+h\sigma)}{\sigma_{n-k}(\n^{2}h+h\sigma)} &= &f^{-1} & \text{ in } \C_\theta,\\ 
	\n_\mu h&=& \cot\theta h & \text{ on } \p \C_\theta.\end{array} \right.
\end{eqnarray} 
where   $h$ is an unknown function on $\C_\theta$, $\n h$ and $\n^2h$ are the gradient and the Hessian of $h$ on $\C_\theta$ w.r.t. the standard spherical metric $\sigma$ on $\C_{\theta}$ resp. 

Guan-Guan \cite[Section~2]{GG02}  transformed the classical prescribed Weingarten curvature problem into solving Eq. \eqref{sup equ} on $\SS^{n}$ and approached it using the degree theory. A key element of their method is the establishment of a priori estimates for solutions of Eq. \eqref{sup equ}. Under a suitable regularity assumption on $f$, Guan-Guan utilized the specific structure of Eq. \eqref{sup equ} to derive positive lower and upper bounds for principal curvatures of the convex hypersurface. Notably, these bounds are independent of $\|h\|_{C^{1}(\SS^{n})}$. Subsequently, they applied Cheng-Yau's Lemma (see \cite[Lemma~3 and Lemma~4]{CY76}) to obtain $C^{0}$ estimate, and the  $C^{1}$ estimate  was obtained via interpolation. However, due to the additional Robin boundary condition, the argument of Guan-Guan does not seem to directly apply to our setting.

In this paper, we adopt an alternative approach to establish a priori estimates for the solution of Eq. \eqref{sup equ} with the Robin boundary value condition.  We briefly outline the proof to conclude this section. Denote $\widehat{\S}$ as the convex body enclosed by $\S$ and $\partial\ov{\RR^{n+1}}$. Let $\rho_{-}(\widehat{\S}, \theta)$ and $\rho_{+}(\widehat{\S}, \theta)$ be the capillary inner and outer radius of convex body $\widehat\S$ respectively (cf. Sinestrari-Weng \cite[Section 2.2]{SW} or Section \ref{sec-2.2}). Let $\lambda\coloneqq(\lambda_{1},\cdots, \lambda_{n})$ be the eigenvalues of $A\coloneqq\n^{2}h+h\s$, namely,  the principal radii of $\S$. For convenience, assume that  $\lambda_{1}\leq \l_2\leq  \cdots\leq \lambda_{n}$.  First, using the special structure of Eq. \eqref{sup equ} and the maximum principle, we derive 
\begin{eqnarray}\label{r upper}
    \rho_{-}(\widehat{\S},\theta)\leq C.
\end{eqnarray}
Next, we establish a quantitative dependence relationship between the maximal eigenvalue  of $A=\n^{2}h+h\s$ (i.e., the maximal principal radii of $\S$) and  $\|h\|_{C^{0}(\C_{\theta})}$ : 
\begin{eqnarray}\label{key est}
   \max\limits_{\C_{\theta}} \lambda_{n}\leq C(1+\|h\|_{C^{0}(\C_{\theta})}),
\end{eqnarray}
where the constant $C$ in \eqref{r upper} and \eqref{key est} depends only on $n, k$ and $f$. This also provides a refined estimate compared to our previous work in \cite[Section 3]{MWW-AIM} for $k=n$ of Eq. \eqref{sup equ}. 
By adapting a geometric lemma of Chou-Wang \cite[Lemma~2.2]{CW00} to the capillary setting, we obtain the following key inequality:
\begin{eqnarray}\label{key geo}
    \frac{\rho_{+}(\widehat{\S},\theta)^{2}}{\rho_{-}(\widehat{\S}, \theta)}\leq C(n)\max\limits_{\C_{\theta}}\lambda_{n},
\end{eqnarray} where $C(n)$ is a positive constant depending only on $n$. Combining \eqref{r upper}, \eqref{key est}, and \eqref{key geo}, we obtain the positive lower and upper bounds for the support function and  $C^2$ estimate simultaneously. Finally, by applying standard degree theory (see, e.g., \cite{GG02} or \cite{GMZ}.), we solve the prescribed Weingarten curvature problem for convex capillary hypersurfaces for capillary even 
functions (Theorem \ref{main-thm}). In particular, this also provides an alternative proof of the results in \cite{GG02}. 

\

\textbf{Organization of the paper.} 
In Section \ref{sec2}, we review the basic properties of elementary symmetric functions and present an important geometric lemma. In Section \ref{sec3}, we derive a priori estimates for the admissible solutions of Eq. \eqref{sup equ}. Section \ref{sec4} is devoted to completing the proof of Theorem \ref{main-thm} using degree theory and applying the inverse function theorem to finalize the proof of Theorem \ref{thm-1.2}. In Section \ref{sec5}, we complete the proof of Theorem \ref{thm-example} by following the strategy presented in \cite[Section~4]{GG02} with minor modifications.

\section{Preliminaries}\label{sec2}
In this section, we summarize some well-known properties of elementary symmetric functions. We then adapt a geometric lemma due to Chou-Wang \cite{CW00} to our capillary setting, which provides a relation between the ratio of the capillary inner and outer radii of the capillary convex body and its principal radii.

\subsection{Elementary symmetric functions}
\begin{defn}
    Let $W=\{W_{ij}\}$ be an $n\times n$ symmetric matrix,   
    \begin{eqnarray*} 
        \sigma_{k}(W)\coloneqq\sigma_{k}(\lambda(W))=\sum\limits_{1\leq i_{1}<i_{2}\cdots< i_{k}\leq n}\lambda_{i_{1}}\lambda_{i_{2}}\cdots \lambda_{i_{k}},~~1\leq k\leq n,
    \end{eqnarray*}
    where $\l\coloneqq\lambda(W)=(\lambda_{1}(W), \lambda_{2}(W), \dots, \lambda_{n}(W))$ is the  set of eigenvalues of  $W$. 

\end{defn}

We use the convention that $\sigma_0=1$ and $\sigma_k =0$ for $k>n$. Let $H_k(\lambda)$ be the normalization of 
$\sigma_{k}(\lambda)$ given by \begin{eqnarray*}
    H_k(\lambda)\coloneqq\frac{1}{\binom{n}{k}}\s_k(\lambda).
\end{eqnarray*} Denote  $\sigma _k (\lambda \left| i \right.)$ the symmetric
	function with $\lambda_i = 0$ and $\sigma _k (\lambda \left| ij \right.)$ the symmetric function with $\lambda_i =\lambda_j = 0$.  Recall that the  G{\aa}rding cone is defined as
\begin{eqnarray*} 
\Gamma_k \coloneqq  \left\{ \lambda  \in \mathbb{R}^n \mid \sigma _i (\lambda ) > 0,~~\forall 1 \le i \le k \right\}.
\end{eqnarray*} 

\begin{defn}
     A function $h\in C^2(\C_\theta)$  is called admissible if $$A\coloneqq\n^2 h(\xi)+h(\xi)\s\in \Gamma_{n}$$
     for all $\xi\in \C_{\theta}$.
\end{defn}
We denote $\sigma _k(W\left|
i \right.)$ the symmetric function with $W$ deleting the $i$-row and
$i$-column and $\sigma _k (W \left| ij \right.)$ the symmetric
function with $W$ deleting the $i,j$-rows and $i,j$-columns. 
\begin{prop}\label{prop2.1}
Suppose that  $W=\{W_{ij}\}$ is diagonal, and $1\leq k\leq n$, 
then
\begin{eqnarray*}
\sigma_{k-1}^{ij}(W)= \begin{cases}
\sigma _{k- 1} (W\left| i \right.), &\text{ if } i = j, \\
0, &\text{ if } i \ne j,
\end{cases}
\end{eqnarray*}
where $\sigma_{k-1}^{ij}(W)\coloneqq \frac{{\partial \sigma _k (W)}} {{\partial W_{ij} }}$.
\end{prop}
\begin{prop}\label{pro-2.3}\ 
    \begin{enumerate}
        \item  For $\lambda \in \Gamma_k$ and $k > l \geq 0$, $ r > s \geq 0$, $k \geq r$, $l \geq s$, 
\begin{eqnarray*}  
\left(\frac{H_{k}(\lambda)}{H_{l}(\lambda)}\right)^{\frac{1}{k-l}}\leq \left(\frac{H_{r}(\lambda)}{H_{s}(\lambda)}\right)^{\frac{1}{r-s}},
\end{eqnarray*}
with equality holds if and only if $\lambda_1 = \lambda_2 = \cdots =\lambda_n >0$.
\item For $0\le l<k\le n$, $\left(\frac{H_k(\lambda)}{H_{l}(\lambda)}\right)^{\frac{1}{k-l}}$ is concave in $\Gamma_
k$.
    \end{enumerate}
\end{prop}
For proofs of Propositions \ref{prop2.1}-\ref{pro-2.3}, see, e.g., \cite[Chapter XV, Section 4]{L96} and \cite[Lemma 2.10, Theorem 2.11]{Spruck}, resp.

\subsection{Chou-Wang's lemma}\label{sec-2.2}
In this subsection, we present a geometric lemma that allows us to control the ratio of the capillary inner and outer radii of a capillary convex body. This result was originally obtained by Chou and Wang for smooth closed strictly convex hypersurfaces, as shown in \cite[Lemma 2.2]{CW00}, see also \cite[Lemma~15.11]{Ben}.  We obtain an analogous result for  strictly convex capillary hypersurfaces in $\ov{\RR^{n+1}_{+}}$. Before delving into the details, we first review some standard notions about the convex body $\wh\S$, which is enclosed by $\S$ and $\partial\ov{\RR^{n+1}}$.
  For the convex body  $\widehat{\S}\subset \ov{\RR^{n+1}_{+}}$, we recall the classical notion of the inner radius of $\widehat{\S}$ 
  \begin{eqnarray*}    
 \rho_{-}(\widehat{\S})\coloneqq\sup\{\rho>0\mid~  B_{\rho}(x_{0})\subset \widehat{\S}~\text{for~some~}x_{0}\in {\RR^{n+1}_{+}}\},\end{eqnarray*} and the outer radius of $\S$ is defined by
 \begin{eqnarray*}    
 \rho_{+}(\widehat{\S})\coloneqq\inf\{\rho>0\mid ~ \widehat{\S}\subset B_{\rho}(x_{0})~\text{for~some~}x_{0}\in {\RR^{n+1}_{+}}\},\end{eqnarray*} 
 where $B_{\rho}(x_0)$ is the ball of radius $\rho$ centered at $x_0$ in $\RR^{n+1}$.
 
By \cite[Section 2.2]{SW}, the capillary inner radius of $\wh\S\subset\ol{\RR^{n+1}_+}$ is defined as
\begin{eqnarray*}
\rho_{-}(\widehat\Sigma, \theta)\coloneqq\sup \left\{r>0 ~\mid~ \widehat{\C_{r,\theta}(x_0)}\subset \widehat\S\text{ for some } x_0\in \p  \ov{\RR^{n+1}_+}\right\},\end{eqnarray*}and 
 the capillary outer radius of $\wh\Sigma$   as
\begin{eqnarray*}
\rho_+(\widehat\Sigma, \theta)\coloneqq \inf \left\{r>0 ~\mid~\widehat\S\subset \widehat{\C_{r,\theta}(x_0)} \text{ for some } x_0\in  \p \ov{\RR^{n+1}_+} \right\},
\end{eqnarray*}where 
\begin{eqnarray*}
    \C_{r,\theta}(x_0)\coloneqq \left\{x\in \ov{\RR^{n+1}_+} ~\mid~ |x-(x_0+r\cos\theta  e)|=r \right\}
\end{eqnarray*} 
is the spherical cap centered at $x_0+r\cos\theta e$ with radius $r>0$. From \cite[Proposition 2.4, (2.26), (2.27)]{SW}, if $\theta\in (0, \frac{\pi}{2}]$, there holds 
\begin{eqnarray}\label{control-rel}
    \rho_{-}(\widehat{\S},\theta)\geq \frac{\rho_{-}(\widehat\S)}{\sin\theta},\quad {\rm{and}}\quad \rho_{+}(\widehat{\S},\theta)\leq \frac{1}{1-\cos\theta}\rho_{+}(\widehat{\S}).
\end{eqnarray}
For a strictly convex capillary hypersurface $\S\subset \ol{\RR^{n+1}_+}$,  the next lemma shows that a uniform bound of the principal radii of $\S$ 
 implies a uniform bound on the ratio of the capillary outer and inner radius of $\wh \S$.
\begin{lem}\label{chou-wang-lemma}
    Let $\S$ be a strictly convex capillary hypersurface in $\ov{\RR^{n+1}_{+}}$ and $\theta\in (0, \frac{\pi}{2}]$. 
    Then there exists a dimensional constant $C>0$ such that
    \begin{eqnarray*}
        \frac{\rho_{+}(\widehat{\S},\theta)^{2}}{\rho_{-}(\widehat{\S},\theta)}\leq C\sup\limits_{x\in  \S}\lambda_{x, \S},
    \end{eqnarray*}
    where $\lambda_{x, \S}$ is the maximal principal radii of $\S$ at the point $x$. 
\end{lem}
\begin{proof}
    The proof is essentially the same as Chou-Wang \cite{CW00} (see also \cite[Lemma 4.1]{HM} for an alternative proof). For the convenience of the reader, we include the proof here with some necessary adjustments to the capillary setting.
  
 According to John's Lemma \cite{John}, there exists an ellipsoid $E\subset\RR^{n+1}$ such that
    \begin{eqnarray}\label{in}
        E\subseteq \widehat{\S}\subseteq n E.
    \end{eqnarray}
    We fix a coordinate system $\{x_{i}\}_{i=1}^{n+1}$ in $\RR^{n+1}$ such that  the ellipsoid $E$ is represented as
    \begin{eqnarray*}
        E\coloneqq \left\{x\in\RR^{n+1}: \frac{x_{1}^{2}}{b_{1}^{2}}+\frac{x_{2}^{2}}{b_{2}^{2}}+\cdots+\frac{x_{n+1}^{2}}{b_{n+1}^{2}}\leq 1 \right\},
    \end{eqnarray*}
   with $0<b_{1}\leq  b_{2}\leq \cdots\leq b_{n+1}$. For a positive constant $b\geq b_{1}$, we define  another ellipsoid $E_{b}$  as 
\begin{eqnarray*}
        E_{b}\coloneqq \left\{x\in\RR^{n+1}: \frac{x_{1}^{2}}{b^{2}}+\cdots+\frac{x_{n}^{2}}{b^{2}}+\frac{x_{n+1}^{2}}{ \left(\frac{b_{n+1}}{\sqrt{2}} \right)^{2}}\leq 1 \right\}.
    \end{eqnarray*}
It is easy to see that $E_{b_{1}}\subseteq E\subseteq \widehat{\S}$. Define 
\begin{eqnarray*}
\widehat{b}\coloneqq\sup \left\{b\in \RR_{+}: E_{b}\subseteq \widehat{\S} \right\},
\end{eqnarray*}
we know that $\wh b$ is well-defined by \eqref{in}. Moreover,   
$\partial E_{\widehat{b}}\cap \S \neq \emptyset$ and $E_{\widehat{b}}\subseteq \widehat{\S}\subseteq n E$.

There must exist a point $x\in\p E_{\wh b}\cap \mathring \S$, otherwise all the touch points lie in $\partial\ov{\RR^{n+1}}$, and it is impossible. 
  Together with the definition of principal radii (or see, e.g., \cite[Lemma~4.3]{HM}), for $x\in \p E_{\wh b}\cap \mathring \S$, we know
\begin{eqnarray}\label{com}
    \lambda_{x, \S}\geq \lambda_{x, \partial E_{\widehat{b}}}.
\end{eqnarray}
It is clear that  $E\subseteq \widehat{\S}$ and $x\in\mathring\S\subset \p \widehat \S$ imply 
\begin{eqnarray}\label{equ-1}
    \frac{x_{1}^{2}}{b_{1}^{2}}+\cdots+\frac{x_{n}^{2}}{b_{n}^{2}}+\frac{x_{n+1}^{2}}{b_{n+1}^{2}}\geq 1.
\end{eqnarray}
From \eqref{in}, we know that $E_{\widehat{b}}\subseteq \widehat{\S}\subseteq nE$, and hence
\begin{eqnarray}\label{equ-2}
    \frac{1}{n}\widehat{b}\leq b_{1}\leq  \cdots\leq b_{n}.
\end{eqnarray}
Denote $x'\coloneqq(x_{1},\cdots, x_{n})$.  Combining \eqref{equ-1} and \eqref{equ-2} we get 
\begin{eqnarray}\label{equ-3}
    \frac{{n^{2}}|x'|^{2}}{\widehat{b}^{2}}+\frac{x_{n+1}^{2}}{b_{n+1}^{2}}\geq 1.
\end{eqnarray}
It is clear that $x\in \partial E_{\widehat{b}}$ means
\begin{eqnarray*}
      \frac{|x'|^{2}}{\widehat{b}^{2}}+\frac{x_{n+1}^{2}}{(\frac{b_{n+1}}{\sqrt{2}})^{2}}= 1, 
\end{eqnarray*}
which follows, together with \eqref{equ-3}, 
\begin{eqnarray}\label{equ-5}
    \frac{|x'|^{2}}{\widehat{b}^{2}}\geq \frac{1}{2n^{2}-1}.
\end{eqnarray}
From \cite[Lemma~4.2]{HM}  and \eqref{equ-5}, we conclude that 
\begin{eqnarray}\label{equ-6}
    \lambda_{x, \partial E_{\widehat{b}}}\geq \frac{|x'|^{3}}{\widehat{b}^{3}}\times \frac{b_{n+1}^{2}}{\widehat{b}}\geq \left(\frac{1}{2n^{2}-1} \right)^{\frac{3}{2}}\frac{b_{n+1}^{2}}{\widehat{b}}\geq  \left(\frac{1}{2n^{2}-1} \right)^{\frac{3}{2}}\frac{n b_{n+1}^{2}}{b_{1}}.
\end{eqnarray}
It is easy to see that $ b_{1}\leq \rho_{-}(\widehat{\S})\leq\rho_{+}(\widehat{\S})\leq n b_{n+1}$. By \eqref{com}, \eqref{equ-6} and \eqref{control-rel}, we have
\begin{eqnarray*}
    \lambda_{x, \S}\geq \lambda_{x, \partial E_{\widehat{b}}}\geq \left(\frac{1}{2n^{2}-1} \right)^{\frac{3}{2}}\frac{\rho_{+}(\widehat{\S})^{2}}{n \rho_{-}(\widehat{\S})}\geq\left(\frac{1}{2n^{2}-1} \right)^{\frac{3}{2}}\frac{\left[(1-\cos\theta)\rho_{+}(\widehat{\S}, \theta)\right]^{2}}{\sin\theta \rho_{-}(\widehat{\S}, \theta)}.
\end{eqnarray*}
Therefore, we complete the proof of Lemma \ref{chou-wang-lemma}.

\end{proof}

\section{A priori estimates}\label{sec3}
The primary objective of this section is to establish the following a priori estimates for solutions of Eq. \eqref{sup equ}, which is the main result of this section.
 \begin{thm}\label{thm est}
    Let $\theta\in (0, \frac{\pi}{2})$, and let $h$ be a positive, capillary even and admissible solution of Eq. \eqref{sup equ}. Then for any $\gamma\in (0, 1)$,  there exists a positive constant $C$ depending only on $n, k, \gamma, \min\limits_{\C_{\theta}}f$ and $\|f\|_{C^{3}(\C_{\theta})}$ such that 
    \begin{eqnarray}\label{global C2}
       \|h\|_{C^{4,\gamma}(\C_{\theta})}\leq C.
    \end{eqnarray}
 \end{thm}
Before proving this theorem, as in \cite{MWW-AIM}, we introduce the capillary support function
    \begin{eqnarray*}
        u\coloneqq\frac{h}{\ell},
    \end{eqnarray*}
    where $\ell\coloneqq\sin^{2}\theta+\cos\theta\<\xi, e\>$ for $\xi\in\C_{\theta}$. By a direct computation, it is easy to see that \eqref{sup equ} is equivalent to  the following equation of $u$, 
\begin{eqnarray}\label{eq of u} \left\{
\begin{array}{rcll}\vspace{2mm}
	\displaystyle \frac{\sigma_{n}(\ell \n^{2}u+\cos\theta(\n u\otimes e^{T}+e^{T}\otimes \n u )+u\sigma)}{\sigma_{n-k}(\ell \n^{2}u+\cos\theta(\n u\otimes e^{T}+e^{T}\otimes \n u)+u\sigma)} &= &f^{-1} & \text{ in } \C_\theta,\\ 
	\n_\mu u&=&0 & \text{ on } \p \C_\theta, \end{array} \right.
\end{eqnarray} which is a Neumann boundary value problem. 

In order to establish Theorem \ref{thm est}, the key ingredient is to derive a quantitative linear dependence between $\max\limits_{\C_{\theta}}|\n^{2}h|$ and $\|h\|_{C^{0}(\C_{\theta})}$, as demonstrated in Lemma \ref{lem-c2 bry} below and ultimately in Theorem \ref{thm C2}. This also provides a refined estimate compared to our previous work in \cite{MWW-AIM}. We then obtain the $C^2$ estimate by combining Theorem \ref{thm C2} with 
Lemma \ref{chou-wang-lemma}, while the higher-order estimates follow from the standard theory of fully nonlinear PDEs with a Neumann boundary condition.

To begin with, we use the special structure of Eq. \eqref{eq of u} and the maximum principle to derive an upper bound for the inner radius.

\begin{lem}\label{lem C0}
 Let $\theta\in (0, \frac{\pi}{2}]$ and $h$ be a positive admissible solution to Eq. \eqref{sup equ}. Then 
    \begin{eqnarray}\label{min-h-upper-bound}
       \min\limits_{\C_{\theta}}h\leq C,
    \end{eqnarray}
    where the positive constant $C$ depends only on $f$.   Furthermore, if $h$ is the capillary even support function of a strictly convex  capillary hypersurface $\S\subset \ol{\RR^{n+1}_+}$, then the inner radii of its convex body $\widehat{\S}$ satisfies \begin{eqnarray}\label{upper bound of r}
        \rho_{-}(\widehat{\S},\theta)\leq C.
    \end{eqnarray}
\end{lem}
\begin{proof}
Suppose the solution $u$ of \eqref{eq of u} attains its minimum value at some point, say $\xi_{0}\in \C_{\theta}$. We consider two cases depending on whether $\xi_0$ is an interior point or lies on the boundary. If $\xi_{0}\in \C_{\theta}\setminus \partial\C_{\theta}$, then 
\begin{eqnarray}\label{deri12}
    \n u(\xi_{0})=0\quad \text{and}\quad \n^{2}u(\xi_{0})  \geq 0.
\end{eqnarray}
If $\xi_{0}\in \partial\C_{\theta}$, let $\{e_i\}_{i=1}^n$ be the orthonormal frame around $\xi_0$ such that $e_n=\mu$. Then the boundary value condition of \eqref{eq of u} implies that $\n u(\xi_0)=0$. Using the minimality of $\xi_0$ again, 
$$\left(\n^2_{\a\b} u(\xi_0) \right)  \geq 0 ~~\text{ for }~~ 1\leq \a,\b\leq n-1.$$  
From  \cite[Proposition 2.8]{MWW-AIM} and $\n u(\xi_0)=0$, 
$$(\n^2 u(\xi_0))(e_\a,e_n)=-\cot\theta u_\a=0, \text{ for } 1\leq \a\leq n-1.$$ Let $\g(t)$ be the geodesic in $\C_\theta$ starting from $\g(0)=\xi_0$ with $\g'(0)=-e_n$ for $t\in[0, \varepsilon]$ and $\varepsilon ~(\varepsilon<\theta)$ sufficiently small, then the minimality of $\xi_0$ implies that $g(t)\coloneqq u(\g(t))$ attains the  minimum value at $t=0$, which implies \begin{eqnarray*}
    0\leq g''(0)=(\n^2 u(\g(0))(\g'(0),\g'(0))+\<\n u(\g(0)), \g''(0)\>.
\end{eqnarray*} Together with $\n u(\xi_0)=0$, we have that   $$(\n^2 u(\xi_0))(e_n,e_n)\geq 0.$$ Hence \eqref{deri12} also holds in the case $\xi_0\in \p \C_\theta$. Therefore, in both cases, 
evaluating the first equation \eqref{eq of u} at $\xi_{0}$ and using \eqref{deri12} imply
\begin{eqnarray*}
u^{k}(\xi_{0})\leq \frac{1}{\min\limits_{\C_{\theta}}f},
\end{eqnarray*} which implies \eqref{min-h-upper-bound}.
Combining the geometric meaning of the capillary support function (cf. \cite[Remark 2.3]{MWW-AIM}) and the definition of $\rho_-(\wh \S,\theta)$, then we conclude that  \eqref{upper bound of r} holds. This completes the proof.

\end{proof}

The following $C^{1}$ estimate for the solution of Eq. \eqref{sup equ} follows directly from the convexity and the $C^0$-estimate.

\begin{lem}\label{C1}
    Let $\theta\in (0, \pi)$ and $h$ be a positive admissible solution to  Eq. \eqref{sup equ}. Then we have
    \begin{eqnarray}\label{gradient ps}
       \max\limits_{\C_{\theta}} |\n h|\leq (1+\cot^{2}\theta)^{\frac{1}{2}}\|h\|_{C^{0}(\C_{\theta})}.
    \end{eqnarray}

\end{lem}
\begin{proof}
    Define the  function
    \begin{eqnarray*}
        P\coloneqq|\n h|^{2}+h^{2}
    \end{eqnarray*}
and suppose that $P$ attains its maximum value at some point, say $\xi_{0}\in \C_{\theta}$. Next, we divide the proof into two cases: either $\xi_0\in \C_\theta\setminus \p \C_\theta$ or $\xi_0\in \p \C_\theta$.  

\begin{itemize}
    \item[\textbf{Case 1}]
 $\xi_{0}\in \C_{\theta}\setminus \partial\C_{\theta}$. By the maximal condition, 
\begin{eqnarray*}
    0=\n_{e_i}P=2h_{k}h_{ki}+2hh_{i},\quad \text{for}~1\leq i\leq n,
\end{eqnarray*}
together with the fact that $\n^2 h+h \s>0$, it follows $\n h(\xi_{0})=0$, then \eqref{gradient ps} holds. 

\item[\textbf{Case 2}] $\xi_{0}\in \partial\C_{\theta}$. We choose an orthonormal frame $\{e_{i}\}_{\alpha=1}^{n-1}$ of $\partial\C_{\theta}$ such that $\{(e_{\alpha})_{\alpha=1}^{n-1}, e_{n}=\mu\}$ forms an orthonormal frame of $\C_{\theta}$.
  From \cite[Proposition 2.8]{MWW-AIM}, we know that $h_{n\alpha}(\xi_{0})=0$.
Using the maximal condition again, we have 
\begin{eqnarray}\label{han}
   0= \n_{e_{\alpha}} P=2\sum\limits_{i=1}^{n}h_{i}h_{i\alpha}+2hh_{\alpha},
\end{eqnarray}
which implies 
\begin{eqnarray*} 
    (\n_{e_{\alpha}}h)(\xi_{0})=0.
\end{eqnarray*}
Combining this with the Robin boundary value condition in \eqref{sup equ}, 
\begin{eqnarray*}
    P\leq h_n^2+h^2= (1+\cot^2\theta) h^2(\xi_0).
\end{eqnarray*}
\end{itemize}
In summary, we know that \eqref{gradient ps} holds and the proof is complete.
\end{proof}

Next, we establish the a priori $C^2$ estimate.  This approach is inspired by the work of Lions-Trudinger-Urbas  \cite{LTU} on the Monge-Amp\`ere equation with a Neumann boundary condition on the uniformly convex domain of Euclidean space, see also \cite{MQ} for the $k$-Hessian equation with a Neumann boundary condition. Following the similar strategy as in \cite{LTU, MQ} and also \cite{MWW-Lp, MWW-AIM}, we divide the proof into two main steps. First, we reduce the global $C^2$ estimate to a boundary double normal $C^2$ estimate, as stated in Lemma \ref{lem-c2 bry} below. Then we derive the boundary double normal $C^2$ estimate by constructing an appropriate test function, as detailed in Lemma \ref{lem-double} below.
In the proof, a precise characterization of the dependence between $\max\limits_{\C_{\theta}}|\nabla^{2}h|$ and $\|h\|_{C^{0}(\C_{\theta})}$ is crucial. This relationship represents a key improvement and refinement over our previous work in \cite[Section 3]{MWW-AIM} for the case $k=n$ of Eq. \eqref{sup equ}.

\begin{lem}\label{lem-c2 bry}
    Let $\theta\in (0, \frac{\pi}{2})$ and $h$ be a positive admissible solution to Eq. \eqref{sup equ}. Then we have 
    \begin{eqnarray}\label{uniform C2}
\max\limits_{\C_{\theta}}|\n^{2}h|\leq \max\limits_{\p \C_{\theta}}|\n^{2}h(\mu,\mu)|+2\|h\|_{C^{0}(\C_{\theta})}+C,
    \end{eqnarray}
    where the positive constant $C$ depends only on $n,k, \min\limits_{\C_{\theta}}f$,  and $\|f\|_{C^{2}(\C_{\theta})}$.
\end{lem}
\begin{proof}
Consider the function 
		\begin{eqnarray*}
			P(\xi, \Xi)\coloneqq\n^{2}h(\Xi, \Xi)+h(\xi),
		\end{eqnarray*}
		for $\xi\in \C_{\theta}$ and the unit vector  $\Xi\in T_{\xi  }\C_{\theta}$. 
		Suppose that $P$ attains its maximum at some   point $\xi_{0}\in \C_{\theta} $ and some unit vector $\Xi_{0}\in T_{\xi_{0}}\C_{\theta}$. 
		We divide the proof into two cases according to whether $\xi_0$ is an interior point or not.
		
		\
		
		{\bf Case  1.} $\xi_0\in \C_\theta \setminus\p \C_\theta$. In this case, we choose an orthonormal frame  $\{e_{i}\}_{i=1}^{n}$ around $\xi_{0}$, such that 
		$A_{ij}\coloneqq h_{ij}+h\s_{ij}$ is diagonal at $\xi_0$ and $\Xi_0=e_1$. 
		Denote \begin{eqnarray}\label{eq3}
			F(A)\coloneqq\left(\frac{\sigma_{n}(A)}{\sigma_{n-k}(A)}\right)^{\frac{1}{k}}=f^{-\frac{1}{k}}\coloneqq\widetilde f,
		\end{eqnarray}  and
		\begin{eqnarray*}
			F^{ij}\coloneqq\frac{\partial F}{\partial A_{ij}},
			\quad  F^{ij,kl}\coloneqq\frac{\partial^{2}F}{\partial A_{ij}\partial A_{kl}},\quad  \mathcal{F}\coloneqq \sum\limits_{i=1}^{n}F^{ii}.
		\end{eqnarray*}
	By the homogeneity of $F$, 
		\begin{eqnarray}\label{sum-1}
			F^{ij}h_{ij}=F^{ii}(A_{ii}-h)=\widetilde{f}-h \mathcal{F}.
		\end{eqnarray}
         By Proposition \ref{pro-2.3} (1), we have
\begin{eqnarray}\label{lower bound}
			\mathcal{F}=\left[\frac{\sigma_{n}(A)}{\sigma_{n-k}(A)}\right]^{\frac{1}{k}-1} \sum_{i=1}^n \frac{\sigma_{n-1}(A|i)\sigma_{k}(A)-\sigma_{n}(A)\sigma_{n-k-1}(A|i)}{\sigma_{n-k}^{2}(A)}\geq \binom{n}{k}^{-\frac{1}{k}},
		\end{eqnarray} 
 Taking the first and second covariant derivatives of Eq. \eqref{eq3} in the $e_1$ direction, it follows that  		\begin{eqnarray}\label{one deri}
			F^{ij}A_{ij 1}= \widetilde f _{1},
		\end{eqnarray}
		and 
		\begin{eqnarray}\label{tw0 deri}
			F^{ij}A_{ij11}=\widetilde f_{11}-F^{ij, kl}A_{ij1}A_{kl1}\geq \widetilde f_{11},
		\end{eqnarray}
        where the last inequality holds due to the concavity of $\left(\frac{\sigma_{n}(A)}{\sigma_{n-k}(A)}\right)^{\frac{1}{k}}$ (see Proposition \ref{pro-2.3}).
		
		From \eqref{one deri} and  \eqref{tw0 deri}, we have	\begin{eqnarray*} F^{ij}h_{ij1}=F^{ij}A_{ij1}-h_{1}\sum\limits_{i=1}^{n}F^{ii}		\geq -\max\limits_{\C_{\theta}}|\n \widetilde f|-\max\limits_{\C_{\theta}}|\n h|\cdot \mathcal{F}
  \end{eqnarray*}
		and
		\begin{eqnarray}\label{sum two deri}		F^{ij}h_{ij11}=F^{ij}A_{ij11}-h_{11}\sum\limits_{i=1}^{n}F^{ii}\geq -\max\limits_{\C_{\theta}}|\n^{2}\widetilde f|-h_{11}\cdot\mathcal{F}.	\end{eqnarray}
		For the standard metric on spherical cap $\C_\theta$,  we have the commutator formulae 
		\begin{eqnarray}\label{third comm}
			h_{kij}=h_{ijk}+h_k\d_{ij}-h_j\d_{ki},
		\end{eqnarray}
		and
		\begin{eqnarray}\label{forth comm}
			h_{klij}=h_{ijkl}+2h_{kl}\d_{ij}-2h_{ij}\d_{kl}+h_{li}\d_{kj}-h_{kj}\d_{il}.
		\end{eqnarray} 
		Inserting \eqref{sum-1} and \eqref{forth comm} into \eqref{sum two deri} yield,  
		\begin{eqnarray}\notag
			F^{ij}h_{11 ij}&=& F^{ij}\left(h_{ij11}+2h_{11}\delta_{ij} -2 h_{ij}+h_{1i}\d_{1j} -h_{1j}\d_{i1}\right)
			\\			&\notag=& F^{ij}h_{ij11}+2h_{11}\mathcal{F}-2F^{ij}h_{ij}\\
			&\geq& (h_{11}+2h)\cdot\mathcal{F}- 2\|\widetilde f\|_{C^{2}(\C_{\theta})}.\label{sum-4}
		\end{eqnarray}
 In view of \eqref{sum-1} and  \eqref{sum-4}, we obtain	\begin{eqnarray*}
			0&\geq &	 F^{ij}P_{ij}=
			F^{ij}h_{11ij}+F^{ij}h_{ij}
			\\&\geq &
			(h_{11}+h) \mathcal{F} -2\|\widetilde f\|_{C^{2}(\C_{\theta})}.
		\end{eqnarray*}
		Together with \eqref{lower bound}, we derive 
		\begin{eqnarray}\label{11}
			|h_{11}(\xi_0)|\leq C+\|h\|_{C^{0}(\C_{\theta})},
		\end{eqnarray}
 where the positive constant $C$ depends only on $n, k$, and $\|\widetilde f\|_{C^{2}(\C_{\theta})}$.
 
		\
		
		{\bf Case 2.}  $ \xi_{0} \in \partial \C_{\theta}$. In this case, we can follow the same argument as in \cite[Proof of Lemma 3.3, Case 2]{MWW-AIM} to obtain 
		\begin{eqnarray}\label{1n}
			\n^2	h(\Xi_0,\Xi_0)  \leq  |h_{\mu\mu}|(\xi_{0})+2 \|h\|_{C^0(\C_{\theta})}.
		\end{eqnarray}
		
        Finally, combining \eqref{11} and \eqref{1n}, we conclude that \eqref{uniform C2} holds.

	\end{proof}

Next, we establish the double normal derivative estimate of $h$ on the boundary, namely the term $|\n^2h(\mu,\mu)|$. To achieve this, we construct two barrier functions of $h_\mu$ near the region of the boundary, and the auxiliary functions here are motivated by \cite{LTU} and also \cite{MQ}.
To begin with, we introduce the function 
	\begin{eqnarray*}
		\zeta(\xi)\coloneqq e^{- d(\xi)}-1,
	\end{eqnarray*}where $d$ is defined as 
 \begin{eqnarray*}
     d(\xi)={\rm{dist}}(\xi, \partial\C_{\theta}),\quad \xi\in \C_{\theta}.
 \end{eqnarray*}
Note that $\zeta(\xi)$ is well-defined for all $\xi\in \C_\theta\setminus \{(1-\cos\theta) E_{n+1}\}$, and this function has been previously used by  Guan \cite[Lemma~3.1]{Gb99} and also in our recent work \cite[Section 3.3]{MWW-AIM}. It is easy to notice that $\zeta|_{\partial \C_{\theta}}=0$ and $\n \zeta|_{\partial \C_{\theta}}=\mu$.  Let $\lambda(\n^{2}\zeta)$ be the eigenvalue of the Hessian matrix of $\zeta$. Near the region of $\partial \C_{\theta}$, it satisfies 
	\begin{eqnarray*}
		e^{d}\lambda (\n^{2}\zeta)= \left(\cot\theta+O(d), \cdots, \cot\theta+O(d), 1 \right).
	\end{eqnarray*}
	In particular, there exists a small constant $\delta_{0}>0$ such that
	\begin{eqnarray}\label{hessian of zeta}
		(	\n^{2}_{ij}\zeta) \geq \frac{1}{2} \min\{\cot\theta,1\}  \s,\quad {\rm{in}}\quad \Omega_{\delta_{0}}.
	\end{eqnarray}
	where $$\Omega_{\delta_{0}}\coloneqq\{\xi\in \C_{\theta}: d(\xi)\leq \delta_{0}\}.$$

 \begin{lem}\label{lem-double}
     Let $\theta\in(0, \frac{\pi}{2})$ and $h$ be a positive admissible solution to Eq. \eqref{sup equ}. Then we have
     \begin{eqnarray*}         \max\limits_{\p\C_{\theta}}|\n^{2}h(\mu,\mu)|\leq C(\|h\|_{C^{0}(\C_{\theta})}+1),
     \end{eqnarray*}
     where the positive constant $C$ depends only on $n, k, \min\limits_{\C_{\theta}}f$ and $\|f\|_{C^{2}(\C_{\theta})}$.
 \end{lem}

\begin{proof}

	We consider an auxiliary function
		\begin{eqnarray*}
			Q(\xi)\coloneqq\<\n h, \n \zeta\>-\left(B+\frac{1}{2}M\right)\zeta(\xi)-\cot\theta h(\xi), \quad  \xi\in \Omega_{\delta_{0}},
		\end{eqnarray*}
		where	$B$ is a positive constant to be determined later.
		
		Assume that $Q $ attains its minimum value at $\xi_{0}\in \left(\Omega_{\delta_{0}}\setminus \partial \Omega_{\delta_{0}}\right)$, and choose an orthonormal frame  $\{e_{i}\}_{i=1}^{n}$ around $\xi_{0}$  such that  $(W_{ij})$ is diagonal at $\xi_{0}$. Using \eqref{third comm}, \eqref{hessian of zeta}, at $\xi_0$, we obtain
		\begin{eqnarray*}
			0&\leq & 	 F^{ij}Q_{ij} 
			\\&=&F^{ij}h_{kij}\zeta_{k}+F^{ij}h_{k}\zeta_{kij}+2F^{ij}h_{ki}\zeta_{kj}-\left(B+\frac{1}{2}M\right)F^{ii}\zeta_{ii}-\cot\theta F^{ii}h_{ii}
			\\&=&F^{ii}\left(A_{iik}-h_{i}\delta_{ki}
			\right)\zeta_{k}+2F^{ii}(A_{ii}-h)\zeta_{ii}-\left(B+\frac{1}{2}M\right)F^{ii}\zeta_{ii}\\
        &&+F^{ii}h_{k}\zeta_{kii}-\cot\theta F^{ii}h_{ii}\\
         &=&\zeta_{k} \widetilde f_{k}-F^{ii}h_{i}\zeta_{i}+2F^{ii}A_{ii}\zeta_{ii}-2hF^{ii}\zeta_{ii}+F^{ii}h_{k}\zeta_{kii}-\cot\theta F^{ii}h_{ii}.
			\\
   && -\left(B+\frac{1}{2}M\right)F^{ii}\zeta_{ii}\\
     &\leq  & C_{1}\left(h+|\n h|+1\right)(\mathcal{F}+1)-\frac{1}{2}\left(B+\frac{1}{2}M\right)\min\{\cot\theta,1\}\mathcal{F},
		\end{eqnarray*}
		where the constant $C_{1}$ depends on $n, k,  \min\limits_{\C_{\theta}}f$ and    $\|f\|_{C^1(\C_{\theta})}$.
		In view of \eqref{lower bound} and \eqref{gradient ps},
		we have
		\begin{eqnarray*} 
			C_{1}(h+|\n h|+1)\left(1+\mathcal{F}\right)-\frac{1}{2}\left(B+\frac{1}{2}M\right)\min\{\cot\theta,1\}\mathcal{F}<0,
		\end{eqnarray*} 
		if $B$ is chosen 
		 by   
		\begin{eqnarray}\label{chosen of A}
			B&=&
\frac{4C_{1} \left(\binom{n}{k}^{\frac{1}{k}}+1 \right)\left[ \left(1+(1+\cot^{2}\theta)^{\frac{1}{2}} \right)\|h\|_{C^{0}(\C_{\theta})}+1\right]}{\min\{\cot\theta,1\}} \notag \\
   &&+\frac{1}{1-e^{-\delta_{0}}}\max\limits_{   \C_{\theta}}(|\n h|+\cot\theta  h). 
		\end{eqnarray} 
		   This  contradicts  $F^{ij}Q_{ij}\geq 0$ at $\xi_0$.  Therefore, $\xi_0\in \p \O_{\d_0}$.
		
		If  $\xi_0\in \partial \C_{\theta}\cap \p \O_{\d_0}$, it is easy to see that $Q(\xi_0)=0$. 
		
		If $\xi_0\in \partial \Omega_{\delta_{0}}\setminus  \partial \C_{\theta} $, we have $d(\xi_0)=\delta_0$, and from \eqref{chosen of A},
		\begin{eqnarray*}
			Q(\xi)\geq -|\n h|+B(1-e^{-\delta_{0}})-\cot\theta h\geq 0.
		\end{eqnarray*}
		
		In conclusion, we deduce that 
		\begin{eqnarray*}
			Q(\xi)\geq 0,\quad{\rm{in}}\quad \Omega_{\delta_{0}}.
		\end{eqnarray*}  \qed
		
		\
		
		Now we are ready to obtain the double normal second derivative estimate of $h$. 
		Assume  $h_{\mu\mu}(\eta_{0})\coloneqq\sup\limits_{\partial \C_{\theta}}h_{\mu\mu}>0$ for some $\eta_0\in \p \C_\theta$. In view of \eqref{han}, \eqref{gradient ps} and $Q\equiv 0$ on $\p \C_\theta$, 
		\begin{eqnarray*}
			0&\geq &Q_{\mu}(\eta_{0})\\
			&\geq &(h_{k\mu}\zeta_{k}+h_{k}\zeta_{k\mu})-\left(B+\frac{1}{2}M\right)\zeta_{\mu}-\cot\theta h_{\mu}\\
			&=& h_{\mu\mu}(\eta_{0})-\left(B+\frac{1}{2}M\right)+h_{k}\zeta_{k\mu}-\cot^{2}\theta h,
		\end{eqnarray*}
		which yields  
		\begin{eqnarray}\label{sup estimate}
			\max\limits_{\partial \C_{\theta}}h_{\mu\mu}\leq C\left(\|h\|_{C^{0}(\C_{\theta})}+1\right)+\frac{1}{2}M,
		\end{eqnarray}
  where $C$ is a  constant depending on $n, k,  \min\limits_{\C_{\theta}}f$ and    $\|f\|_{C^1(\C_{\theta})}$.  
To derive the lower bound of $h_{\mu\mu}$ at $\p\C_\theta$, we consider an auxiliary function as
		\begin{eqnarray*}
			\ov{Q}(\xi)\coloneqq\<\n h, \n \zeta\>+\left(\ov{B}+\frac{1}{2}M\right)\zeta(\xi)-\cot\theta h, \quad \xi\in \Omega_{\delta_{0}},
		\end{eqnarray*}
		where 	$\ov{B}>0$ is a positive constant. Similarly as above, we get $$\ov{Q}(\xi)\leq 0 ~\text{ in }   \Omega_{\delta_{0}},$$ and further
		\begin{eqnarray}\label{inf estimate}
			\min\limits_{\partial \C_{\theta}}h_{\mu\mu}\geq -C\left(\|h\|_{C^{0}(\C_{\theta})}+1\right)-\frac{1}{2}M.
		\end{eqnarray}
Therefore, \eqref{sup estimate} and \eqref{inf estimate} together yield
		 \begin{eqnarray*}\max\limits_{ \p\C_{\theta}}|\n^{2}h(\mu,\mu)|\leq C(\|h\|_{C^{0}(\C_{\theta})}+1).
     \end{eqnarray*}

\end{proof}

As a direct consequence of  Lemma \ref{lem-c2 bry} and Lemma \ref{lem-double}, we obtain the following quantitative dependence relationship between $\max\limits_{\C_{\theta}}|\n^{2}h|$ and $\|h\|_{C^{0}(\C_{\theta})}$,  which plays an important role in our subsequent analysis.
\begin{thm}\label{thm C2}
    Let $\theta\in (0, \frac{\pi}{2})$ and $h$ be a positive admissible solution to Eq. \eqref{sup equ}. Then we have
    \begin{eqnarray*}
        \max\limits_{\C_{\theta}}|\n^{2}h|\leq C(1+\|h\|_{C^{0}(\C_{\theta})}),
    \end{eqnarray*}
    where the positive constant $C$ depends on $n, k, \min\limits_{\C_{\theta}}f$ and $\|f\|_{C^{2}(\C_{\theta})}$.  Moreover, there exist positive constants $c_0, C_0$ depending only on $n,k, \min\limits_{\C_\theta} f$ and $\|f\|_{C^2(\C_\theta)}$, such that 
    \begin{eqnarray}\label{principal-radii-estimate}
 \frac{c_0}{ \left(1+\|h\|_{C^0(\C_\theta)}\right)^{n-1}}\sigma  \leq   \n^2 h+h\sigma \leq C_0 \left(1+\|h\|_{C^0(\C_\theta)} \right)\sigma, ~~~ \text{ in } \C_\theta.
    \end{eqnarray}
\end{thm}
\begin{proof}
    The upper bound in \eqref{principal-radii-estimate} follows directly from Lemma \ref{lem-c2 bry} and Lemma \ref{lem-double}. We only derive the lower bound. By Proposition \ref{pro-2.3} (1) and Eq. \eqref{sup equ},  
    \begin{eqnarray*}
        \s_n(\n^2h+h\s)=f^{-1} \s_{n-k}(\n^2 h+h \s) \geq f^{-1} \binom{n}{k} \s_n(\n^2 h+h\s)^{\frac{n-k}{n}},
    \end{eqnarray*}it yields
    \begin{eqnarray*}
        \s_n(\n^2 h+h\s) \geq  \left(\binom{n}{k} f^{-1}\right)^{\frac n k},
    \end{eqnarray*} together with the upper bound in \eqref{principal-radii-estimate}, then the lower bound of \eqref{principal-radii-estimate} follows.
\end{proof}

Now we proceed to prove Theorem \ref{thm est}.
\begin{proof}[\textbf{Proof of Theorem \ref{thm est}}]
From Lemma \ref{chou-wang-lemma}, Lemma \ref{lem C0} and  Theorem \ref{thm C2}, we have
     \begin{eqnarray*}
          \rho_{+}(\widehat{\S}, \theta)^{2}\leq C\rho_{-}(\widehat{\S}, \theta)\max\limits_{\C_{\theta}}\lambda_{n}\leq C\left(1+\rho_{+}(\widehat{\S},\theta)\right),
     \end{eqnarray*}
     and hence $$\rho_{+}(\widehat{\S}, \theta)\leq C.$$ Together with  Lemma \ref{C1}
 and Theorem \ref{thm C2}, we conclude that
 \begin{eqnarray*} 
     \|h\|_{C^{2}(\C_{\theta})}\leq C.
 \end{eqnarray*}
By applying the theory of fully nonlinear second-order uniformly elliptic equations with oblique derivative boundary conditions (cf. \cite[Theorem 1.1]{LT}), then we obtain $C^{2,\gamma}$ $(\gamma \in (0,1))$ estimate of $h$, along with the higher-order estimates as in \eqref{global C2}. This completes the proof.
 \end{proof}

\section{Proof of Theorems \ref{main-thm} and \ref{thm-1.2}}\label{sec4}
In this section, we use a degree theory argument as in \cite{LYY, LLN} to complete the proof of Theorem \ref{main-thm}, following the approach of \cite{GG02, GMZ, GZ}. To compute the degree, we require a uniqueness result for capillary convex hypersurfaces $\Sigma\subset\ol{\RR^{n+1}_+}$ with constant $k$-th Weingarten curvature and constant contact angle $\theta\in(0, \pi)$ along $\p\S\subset\p \RR^{n+1}_+$, which was recently established in \cite[Corollary 1.2]{JWXZ}. In turn, this implies that the function $$ \binom{n}{k}^{\frac1 k}\ell \eqqcolon \wt\ell$$ is the unique solution to Eq. \eqref{sup equ} if $f\equiv 1$. We now prove Theorem \ref{main-thm} as follows.

\begin{proof}[\textbf{Proof of Theorem  \ref{main-thm}}]
For $\gamma\in (0,1)$ and integer $l\geq 0$, set
\begin{eqnarray*}
    \mathcal{A}^{l,\gamma}\coloneqq \left\{h\in C^{l,\gamma}(\C_{
    \theta
    }): h(\xi)=h(\widehat{\xi}),~\forall\xi\in \C_{\theta}  \right\}.
\end{eqnarray*}
Now define a subset of $\mathcal{A}^{l,\gamma}$ 
\begin{eqnarray*}
    \mathcal{B}\coloneqq  \left\{h\in \mathcal{A}^{l+2,\gamma}: \n^{2}h+h\sigma >0,~\n_{\mu}h=\cot\theta h~{\rm{on}}~\partial\C_{\theta}~{\rm{and}}~\|h\|_{C^{l+2,\gamma}(\C_{\theta})}<C \right\},
\end{eqnarray*}
where $C$ is a uniform positive constant to be determined later.

Consider the map $G(\cdot, t): \mathcal{A}^{l+2,\gamma}\rightarrow\mathcal{A}^{l, \gamma}$ given by
\begin{eqnarray*}
G(h, t)=\frac{\sigma_{n}(\n^{2}h+h\sigma)}{\sigma_{n-k}(\n^{2}h+h\sigma)}-f^{t},
\end{eqnarray*}
where 
$$f^{t}\coloneqq (1-t)\frac{1}{\binom{n}{k}}+t f^{-1}.$$
We next \textbf{Claim} that if $C$ is sufficiently large, then $G(h,t)=0$ has no solution on the boundary of $\mathcal{B}$, i.e., $\partial \mathcal{B}$.  In fact, we argue by contradiction, if not, there exists a function $h\in \partial \mathcal{B}$, i.e., $\n^{2}h+h\sigma\geq 0$ or $\|h\|_{C^{4,\gamma}(\C_{\theta})}=C$, satisfying
$$
\frac{\sigma_{n}(A)}{\sigma_{n-k}(A)}= f^{t}.
$$
By Theorem \ref{thm est} and $f^{t}$ is a positive function, this reaches a contradiction. Thus \textbf{Claim} is true. Consequently, by \cite[Proposition~2.2]{LYY}, we have
\begin{eqnarray}\label{deg-1}
    \text{deg}\left(G(\cdot, 0), \mathcal{B}, 0\right)=\text{deg}(G(\cdot, 1), \mathcal{B}, 0).
\end{eqnarray}
From \cite[Corollary~1.2]{JWXZ}, we have that  $h_{0}=\ell$ is the unique capillary even solution to $G(h, 0)=0$ and  the linearized operator of $G$ at $h_{0}=\ell$ is
  \begin{eqnarray}\label{L-operator}
      \mathcal{L}\varphi=a_{0} (\Delta\varphi+n \varphi),\quad  \text{for~any}~\varphi\in \mathcal{A}^{l+2,\gamma},
  \end{eqnarray}
where $a_{0}$ is a positive constant.   If $\varphi\in \text{Ker}(\mathcal{L})\cap \mathcal{A}^{l+2,\gamma}$, which implies $\Delta\varphi+n\varphi=0$, then $\varphi= \sum \limits_{\alpha=1}^{n}a_{\alpha}\<\xi, E_{\alpha}\>$ for some constants $\{a_{\alpha}\}_{\alpha=1}^{n}\subset \RR$ (or see \cite[Lemma 4.3]{MWW-AIM}). On the other hand, $\varphi$ is a capillary even function, it implies
\begin{eqnarray*}
    \int_{\C_{\theta}}\varphi \<\xi, E_{\alpha}\> dA_\sigma =0,\quad \text{for}~1\leq \alpha\leq n,
\end{eqnarray*}
then we conclude that $\varphi=0$  and  in turn the operator 
 $\mathcal{L}$ is invertible. It follows that (cf. \cite[Proposition~2.3, Proposition~2.4]{LYY} and \cite[Theorem~1.1]{LLN}), 
\begin{eqnarray}\label{deg-2}
    \text{deg}(G(\cdot, 0), \mathcal{B}, 0)=\text{deg}(\mathcal{L}, \mathcal{B}, 0)=\pm 1.
\end{eqnarray}
Combining \eqref{deg-1} and \eqref{deg-2}, we conclude that $\text{deg}(G(\cdot, 1), \mathcal{B}, 0)\neq 0$,  which implies Eq. $F(\cdot, 1)=0$, i.e., Eq. \eqref{sup equ} has at least one solution in $\mathcal{B}$ and  we complete the proof of Theorem \ref{main-thm}.

\end{proof}

Next, we proceed to prove Theorem \ref{thm-1.2}. In fact, this result is a direct consequence of  Theorem \ref{thm est} and the inverse function theorem. We start with the following result. 

\begin{lem}
    Let $f$ be a positive smooth, capillary even function on $\C_{\theta}$. Suppose there exists a small positive  constant $\varepsilon_{0}$ such that $\|f^{-1}-1\|_{C^{\alpha}(\C_{\theta})}\leq \varepsilon_{0}$ and two capillary even functions $h_{1}, h_{2}$ satisfy Eq. \eqref{sup equ}. If $\|h_{i}-\wt\ell\|_{C^{0}(\C_{\theta})}\leq \varepsilon_{0}$ for $i=1,2$, then $h_{1}=h_{2}$. 
\end{lem}

\begin{proof}
  Direct calculations yield
    \begin{eqnarray}
        f^{-1}-1&=&\widehat F(\n^{2}h_{1}+h_{1}\sigma)-\widehat F(\n^{2}\wt\ell+\wt\ell \sigma)\notag \\
        &=&\int_{0}^{1}\frac{d}{dt}\widehat F(\n^{2}((1-t)\wt\ell+th_{1})+((1-t)\wt\ell+th_{1})\sigma)dt\notag \\
        &\coloneqq&\sum\limits_{i,j=1}^{n}a^{ij}\left[(h_{1}-\wt\ell)_{ij}+(h_{1}-\wt\ell)\sigma_{ij}\right], \label{elliptic equ}
    \end{eqnarray}
    where $\widehat F(A)=\frac{\sigma_{n}(A)}{\sigma_{n-k}(A)}, \widehat F^{ij}=\frac{\partial \widehat F(A)}{\partial A_{ij}}$ and $
    a^{ij}=\int_{0}^{1}\widehat F^{ij}(\n^{2}((1-t)\wt\ell+th_{1})+((1-t)\wt\ell+th_{1})\sigma)dt$. 

    By Theorem \ref{thm est}, there exists a positive constant $C$ such that $\frac{1}{C} I\leq \{a^{ij}\}\leq C I$, and hence Eq. \eqref{elliptic equ} is uniformly elliptic.
   The standard  nonlinear elliptic theory with Oblique boundary value condition  (cf. \cite[Theorem~1]{LT}) implies 
   \begin{eqnarray*}
       \|h_{1}-\wt\ell\|_{C^{2, \alpha}(\C_{\theta})}\leq C \left(\|h-\wt\ell\|_{C^{0}(\C_{\theta})}+\|f^{-1}-1\|_{C^{\alpha}(\C_{\theta})} \right)\leq 2 C\varepsilon_{0}.
   \end{eqnarray*}
   Similarly to above, we have
   \begin{eqnarray*}
       \|h_{2}-\wt\ell\|_{C^{2, \alpha}(\C_{\theta})}\leq 2C\varepsilon_{0}.
   \end{eqnarray*}
   By the inverse function theorem and the linearized operator of Eq. \eqref{sup equ} is invertible at $h=\wt\ell$ (up to a positive constant scaling of $\ell$), thus it follows that $h_{1}=h_{2}$ if $\varepsilon_{0}$ is sufficiently small.   
\end{proof}

\begin{proof}[\textbf{Proof of Theorem \ref{thm-1.2}}]

Suppose that Theorem \ref{thm-1.2} is not true. Then there exist two sequence of functions $\{f_{i}\}_{i\in \NN}, \{h_{i}\}_{i\in \NN}$, such that 
\begin{eqnarray*} \left\{
\begin{array}{rcll}\vspace{2mm}\displaystyle
	 \frac{\sigma_{n}(\n^{2}h_i+h_i\sigma)}{\sigma_{n-k}(\n^{2}h_i+h_i\sigma)} &= &f_i^{-1} & \text{ in } \C_\theta,\\ 
	\n_\mu h_i&=& \cot\theta h_i & \text{ on } \p \C_\theta,\end{array} \right.
\end{eqnarray*} 
and if $i\rightarrow \infty$, satisfying 
\begin{eqnarray*}
    \|f_{i}^{-1}-1\|_{C^{\alpha}(\C_{\theta})}\leq \frac{1}{i},\quad {\rm{and}}\quad \|h_{i}-\wt\ell\|_{C^{0}(\C_{\theta})}\geq \varepsilon_{0}.
\end{eqnarray*}
From Theorem \ref{thm est}, after passing to a subsequence, $h_{i}$ converges to some function $h_\infty\in C^{2}(\C_{\theta})$ and $h_\infty$ satisfies
\begin{eqnarray}\label{sup equ-1} \left\{
\begin{array}{rcll}\vspace{2mm}\displaystyle
	 \frac{\sigma_{n}(\n^{2}h_\infty+h_\infty\sigma)}{\sigma_{n-k}(\n^{2}h_\infty+h_\infty\sigma)} &= &1 & \text{ in } \C_\theta,\\ 
	\n_\mu h_\infty&=& \cot\theta h_\infty & \text{ on } \p \C_\theta,\end{array} \right.
\end{eqnarray} 
and 
\begin{eqnarray*}
    \|h_\infty-\wt\ell\|_{C^{0}(\C_{\theta})}\geq \varepsilon_{0}.
\end{eqnarray*}
This contradicts the fact that $\wt\ell$ is the unique strictly convex solution to Eq. \eqref{sup equ-1} with $f=1$ (cf. \cite[Corollary~1.2]{JWXZ}). Hence, we complete the proof.  
\end{proof}

\section{Counterexample}\label{sec5}
In this section, we construct some counterexamples to show the condition \eqref{nece suff cond} is not a necessary condition for the solvability of Eq. \eqref{cur equ}. We follow the strategy presented in \cite[Section~4]{GG02} with only minor modifications.

Let $v\in C^{\infty}(\C_{\theta})$ and $\n_{\mu}v=\cot\theta v$ on $\partial \C_{\theta}$. For small $t>0$, the function $$h_{t}\coloneqq \ell+tv,$$ is the support function of some smooth, strictly convex capillary hypersurface in $\ol{\RR^{n+1}_+}$ (cf. \cite[Proposition 2.6]{MWWX}), and further
\begin{eqnarray*}
    H_{n}(\n^{2}h_{t}+h_{t}\sigma)=\sum\limits_{i=0}^{n}\frac{n!}{i!(n-i)!}H_{i}t^{i},
\end{eqnarray*}
where  $H_{i}\coloneqq\frac{i!(n-i)!}{n!}\sigma_{i}(\n^{2}v+v\sigma)$. The direct computation implies
\begin{eqnarray}\label{zero-1}
    \int_{\C_{\theta}}H_{i}(\n^{2}v+v\sigma) \<\xi, E_{\alpha}\>dA_\sigma=0,\quad {\rm{for~all}}\quad 1\leq i, \alpha \leq n. 
\end{eqnarray} (See also \cite[Proposition~2.6]{MWW-CM}).
For a  fixed integer $k ~(1\leq k\leq n)$,  the following expansion formula holds
\begin{eqnarray}\label{expansion}
    \frac{H_{n}(\n^{2}h_{t}+h_{t}\sigma)}{H_{n-k}(\n^{2}h_{t}+h_{t}\sigma)}=1+a_{1}t+a_{2}t^{2}+O(t^{3}),
\end{eqnarray}
where 
\begin{eqnarray*}
    a_{1}\coloneqq (n-k)H_{1},\quad \quad
    a_{2}\coloneqq \frac{n-k}{2}\left[(n+k-1)H_{2}-2kH_{1}^{2}\right].
\end{eqnarray*}

We conclude this paper by completing the proof of Theorem \ref{thm-example}.


\begin{proof}[\textbf{Proof of Theorem  \ref{thm-example}}]
    Using the spherical coordinates on $\C_{\theta}$, we have 
    \begin{eqnarray*}
        &&\xi_{n+1}=\cos\theta_{1}-\cos\theta,\\
        &&\xi_{j}=\sin\theta_{1}\cdots \sin\theta_{n-j+1} \cos\theta_{n-j+2},\quad 2\leq j\leq n, \\
        &&\xi_{1}=\sin\theta_{1}\sin\theta_{2}\cdots \sin\theta_{n}, \\    
        &&d A_\sigma= \sin^{n-1}\theta_{1}\sin^{n-2}\theta_{2}\cdots \sin\theta_{n-1}d\theta_{n}\cdots d\theta_{1},
    \end{eqnarray*}
where $\theta_{1}\in [0,  \theta], \theta_{n}\in [0, 2\pi]$ and $0\leq \theta_{j}\leq \pi$ for all $2\leq i\leq n-1$.
We modify the function $g$ chosen  in \cite[Proposition~4.1]{GG02} as follows, 
\begin{eqnarray*} 
    g(\xi)=\eta(\cos^{2}\theta_{2})\cdots \eta(\cos^{2}\theta_{n-1})(\cos 2\theta_{n}+\sin3 \theta_{n}),
\end{eqnarray*}
where $\eta(t)$ is a smooth cut-off function satisfying $0\leq \eta\leq 1$ and $\eta=1$ if $|t|<\frac{1}{2}$, and $\eta=0$ if $|t|>\frac{3}{4}$. It is easy to check that 
\begin{eqnarray}\label{per-Ker}
    \int_{\C_{\theta}}\<\xi, E_{\alpha}\> g(\xi)dA_\sigma=0, \quad {\rm{for~all}}~1\leq \alpha \leq n,
\end{eqnarray}
and 
\begin{eqnarray}\label{non-zero}
    \int_{\C_{\theta}}\<\xi, E_{1}\>g^{2}(\xi)dA_\sigma\neq 0.
\end{eqnarray}
By \eqref{per-Ker}, we know that $g\in \left({\rm{Ker}(\mathcal{L})}\right)^{\perp}$, where the elliptic operator $\mathcal{L}$ is given by \eqref{L-operator}. By the Fredholm alternative theorem implies that there exists $v\in C^{\infty}(\C_{\theta})$ satisfying
\begin{eqnarray*} \left\{
\begin{array}{llll}
	H_{1}=\frac{1}{n}(\Delta v+nv)&= &g & \text{ in } \C_\theta,\\ 
	\n_\mu v&=& \cot\theta v & \text{ on } \p \C_\theta.\end{array} \right.
\end{eqnarray*}
Then, for all sufficiently small $t>0$, we have $\n^{2}h_{t}+h_{t}\sigma>0$. Combining \eqref{zero-1}, \eqref{expansion}, \eqref{per-Ker} and \eqref{non-zero},  we conclude that 
\begin{eqnarray*} 
    \int_{\C_{\theta}}\frac{H_{n}(\n^{2}h_{t}+h_{t}\sigma)}{H_{n-k}(\n^{2}h_{t}+h_{t}\sigma)}\left\<\xi, E_{1}\right\>dA_{\sigma}=-k(n-k)t^{2}\int_{\C_{\theta}}\<\xi, E_{1}\>g^{2}dA_{\sigma}+O(t^{3})\neq 0.
    \end{eqnarray*}
This completes the proof.

\end{proof}

\

\subsection*{Acknowledgment}  X.M. was supported by  the National Key R$\&$D Program of China 2020YFA0712800  and  the Postdoctoral Fellowship Program of CPSF under Grant Number GZC20240052.  L.W. was partially supported by CRM De Giorgi of Scuola Normale Superiore and PRIN Project 2022E9CF89 of University of Pisa.

\ 

\bigskip



\printbibliography 

\end{document}